\documentclass[a4paper,12pt]{amsart}
\usepackage{amsmath,amsthm,amssymb}
\usepackage{hyperref}

\allowdisplaybreaks[1]

\newcommand{\ity}{\infty}
\newcommand{\C}{\mathbb{C}}

\newcommand{\N}{\mathbb{N}}

\newcommand{\F}{\mathfrak{F}}
\newcommand{\G}{\mathfrak{G}}
\newcommand{\A}{\mathfrak{A}}
\newcommand{\D}{\mathbb{D}}

\numberwithin{equation}{section}

\newtheorem{theorem}{Theorem}[section]
\newtheorem{lemma}[theorem]{Lemma}

\theoremstyle{remark}
\newtheorem{remark}[theorem]{Remark}
\newtheorem{example}[theorem]{Example}
\newtheorem{definition}[theorem]{Definition}

\thanks {The research work of the first  author is supported by research fellowship from Council of Scientific and Industrial Research (CSIR), New Delhi.}

\begin{document}

\title[normality and schwarzian derivatives]{A note on Schwarzian derivatives and normal families}
\author[D. Kumar]{Dinesh Kumar}
\address{Department of Mathematics, University of Delhi,
Delhi--110 007, India}

\email{dinukumar680@gmail.com }

%

\author[S. Kumar]{Sanjay Kumar}

\address{Department of Mathematics, Deen Dayal Upadhyaya College, University of Delhi,
Delhi--110 007, India }

\email{sanjpant@gmail.com}

\begin{abstract}
We establish a criterion for local boundedness and hence normality of a family $\F$ of analytic functions on a domain $D$ in the complex plane whose corresponding family of derivatives is locally bounded. Furthermore we investigate the relation between domains of normality of a family $\F$ of meromorphic functions and its corresponding Schwarzian derivative family. We also establish some criterion for the Schwarzian derivative family of a family $\F$  of analytic functions on a domain $D$ in the complex plane to be a normal family.
\end{abstract}


\maketitle

\section{Introduction}\label{sec1}
A family $\F$ of analytic functions on a domain $\Omega\subset\C$ is locally bounded on $\Omega$  if it is uniformly bounded on each compact subset of $\Omega.$ Equivalently, $\F$ is locally bounded on $\Omega$ if it is uniformly bounded in a neighborhood of each point of $\Omega$ (see \cite{complex analysis},~\cite{schiff}). By Marty's theorem (see \cite{complex analysis}), a family $\F$ of meromorphic functions on a domain $\Omega\subset\C$ is normal if and only if the family $\F^\#=\{f^\#: f\in\F\}$ of the corresponding spherical derivatives is locally uniformly bounded in $\Omega,$ where $f^\#=\frac{|f'|}{1+|f|^2}.$

It is well known ~\cite{schiff}, for a family $\F$ of locally bounded analytic functions on a domain $D$, the corresponding family of derivatives  $\F'=\{f' : f\in\F\}$ also forms a locally bounded family in $D$. Its converse is not true in full generality, however some partial converse ~\cite{schiff} holds. We also establish a partial converse to this result. 
To our best knowledge, 
it hasn't been studied so far what can be said about the relation between the domains of normality of a family $\F$ of meromorphic functions, its conjugate family and the corresponding Schwarzian derivative (SD for short) family. The aim of the present paper is to tackle these questions. We also establish a criterion for the SD family of a family $\F$ of analytic functions on a domain $D\subset\C$ to be a  normal family.

For a meromorphic function $f$ on $\C$, its SD is defined by $S_{f}(z)=\frac{f'''(z)}{f'(z)}-\frac{3}{2}(\frac{f''(z)}{f'(z)})^2,$ which has the invariance property $S_{\tau\circ f}=S_f$ for every Mobius transformation $\tau(z)=\frac{az+b}{cz+d},\, ad-bc\neq 0.$ It is well known \cite{lehto} $S_f=0$ if and only if $f$ is a Mobius transformation. Also if $f$ is a meromorphic function on $\C$ and $f(z)\neq 0$ for all $z\in\C,$ then $S_f=S_{\frac{1}{f}}.$ We show this fact is true even if the meromorphic function $f$ omits a finite value. For  meromorphic functions $f$ and $g$ for which the composition $g\circ f$ is defined, $S_{g\circ f}=(S_g\circ f){f'}^2+S_f$ holds. As $S_\tau=0$ for each Mobius transformation $\tau, S_{g\circ\tau}=(S_g\circ\tau){\tau'}^2.$

Recently, Steinmetz \cite{steinmetz} gave a completely different proof of a normality criterion involving spherical derivatives given by Grahl and Nevo \cite{grahl}. The proof was based on a property equivalent to $f^\#(z)>0,$ which is again equivalent to the fact that corresponding SD is holomorphic on the given domain.

\section{Theorems and their proofs}\label{sec2}
\begin{theorem}\label{sec2,thm1}
Let $\F$ be a family of analytic functions on the unit disk $\D$ such that $f(0)=0$ for all $f\in\F.$ Suppose $\F'$ is locally bounded on $\D.$ Then the family $\F$ is locally bounded on $\D.$
\end{theorem}
\begin{proof}
As $\F'$ is locally bounded, therefore  $\F$ is equicontinuous on each compact subset of $\D.$ We now show  $\F$ is pointwise bounded on $\D$ which will imply  $\F$ is locally bounded on $\D.$ Let $z\in\D$ be arbitrary. For each $f\in\F$ integrating $f$ along the straight line segment $L$ joining $0$ to $z$ we have
\begin{align*}
|f(z)-f(0)| & =|\int_{0}^{z}f'(z)dz|\\
                      &\leq K|z|
\end{align*}
for some $K>0$ depending on the line segment $L.$ Now consider any disk about the point $z$ and integrating along any straight line segment in the disk one gets a similar inequality and the result follows.
\end{proof}
\begin{remark}\label{sec2,rem1}
The above result is true for any arbitrary domain $D\subset\C$ and for any $z_0\in D$ which is common fixed point of the analytic family $\F.$
\end{remark}
We now establish a theorem which gives a criterion for the SD family of a family $\F$ of analytic functions on a domain $D\subset\C$ to be a normal family.
\begin{theorem}\label{sec2,thm2}
Let $\F$ be a family of analytic functions on a domain $D\subset\C$ such that the family of its derivatives $\F'$ satisfies:
\begin{enumerate}
\item\ $|f'|\geq\epsilon$ for all $f\in\F$ and for some fixed $\epsilon>0$\\
\item\ The family $\F'$ is locally bounded.
\end{enumerate}
Then the SD family is normal on $D.$
\end{theorem}
\begin{proof}
$|f'|>0$ is equivalent to local univalence of the function $f.$ As the derivative of a locally bounded family of analytic functions is locally bounded (see \cite{complex analysis}), therefore the families $\F'' $ and $\F'''$ are locally bounded. For $f\in \F,\,S_f(z)=\frac{f'''(z)}{f'(z)}-\frac{3}{2}(\frac{f''(z)}{f'(z)})^2$ and using the local boundedness of the families $\F'' $ and $\F''',$ one obtains the SD family is locally bounded on $D$ and so is normal on $D$ by Montel's theorem.
\end{proof}
We now show if a meromorphic function $f$ omits a finite  value then $S_f=S_{\frac{1}{f}}.$\\
\begin{theorem}\label{sec2,thm3}
For a meromorphic function $f$ on $\C$ which omits a finite value,  $S_f=S_{\frac{1}{f}}.$
\end{theorem}
\begin{proof}
Suppose $f$ omits a value $w\in\C.$ Define a new function $g(z)=f(z)-w.$ Then $g$ omits $0$. Also $S_g=S_{\frac{1}{g}}, S_g=S_{f-w}=S_f$ and $S_{\frac{1}{g}}=S_{\frac{1}{f-w}}=S_{\frac{1}{f}},$ the result follows.
\end{proof}
We wanted to analyse the  relation  between the domains of normality of a family $\mathfrak{F}$ of meromorphic functions on the complex plane $\C$ and that of its  SD family. We observe as such no correlation between them and tried to substantiate this through multiple examples. 
 In all the subsequent examples to follow the domain of the family of meromorphic functions under consideration is the complex plane $\C.$
\begin{example}\label{eg1}
$f_n(z)=e^{nz}.$ Its domain of normality is $\C\setminus\{z:Re(z)=0\}.$ The SD family is $S_{f_n}(z)=\frac{-n^2}{2}$ which is normal in $\C.$
\end{example}
\begin{example}\label{eg2}
$f_n(z)=e^{\frac{z}{nz+1}}.$ The domain of normality of this family is the complex plane $\C.$ The SD family is $S_{f_n}(z)=\frac{-1}{2(nz+1)^4}$ which is normal in the punctured complex plane $\C\setminus{0}.$
\end{example}
However there is possibility of similar domains of normality as shown by the following example.
\begin{example}\label{eg3}
$f_n(z)=e^z-n.$ Domain of normality is $\C.$ The SD family is $S_{f_n}(z)=\frac{-1}{2}$ which is also normal in $\C.$
\end{example}

Under some condition on the family $\F$ of meromorphic functions, we  show its domain of normality is contained in the domain of normality of the corresponding SD family. We will need the following lemmas .
\begin{lemma}\cite{lehto}\label{sec2,lemma1}
Consider  a family $\F$ of locally injective meromorphic functions on a domain $D\subset\C$ which converges locally uniformly to a locally injective meromorphic function $f$ on $D$, then the corresponding SD family of $\F$ converges locally uniformly on $D$ to $S_f$.
\end{lemma}
\begin{proof}
Using the local injectiveness of both the family $\F$ and the limit function $f$ and  local uniform convergence of $\F$ on $D$ to $f$, we get the corresponding SD family of $\F$ converges locally uniformly on $D$ to $S_f$ and the result gets proved.
\end{proof}
\begin{lemma}\cite{schiff}\label{sec2,lemma2}
Let $\F$ be a normal family of analytic functions on a domain $D\subset\C.$ If for all $f\in\F, |f(\zeta)|\leq L$, for some point $\zeta\in D$ and for some $L<\ity,$ then $\F$ is locally bounded.
\end{lemma}
\begin{theorem}\label{sec2,thm4}
Let $\F$ be a family of locally injective meromorphic functions on $\C$. Let $D$ be the domain of normality of $\F$. Then  $D$ is contained in the domain of normality of the corresponding SD family if the limit function of $\F$ on $D$ is locally injective.
\end{theorem}
\begin{proof}
Let $f$ be the limit function of $\F$ on $D$. Using Lemma~\ref{sec2,lemma1},  the SD family of $\F$ converges locally uniformly on $D$ to $S_f$. By definition $D$ is contained in the domain of normality of the SD family.
\end{proof}

\begin{theorem}\label{sec2,thm5}
Let $\F$ be a family of analytic functions on $\C.$ Let $D$ be the domain of normality of $\F.$ Suppose $\F$ satisfies following conditions on $D$:
\begin{enumerate}
\item\ For some point $\zeta\in D$ and for all $f\in\F, |f(\zeta)|\leq L$ for some   $L<\ity,$\\
\item\ $|f'|\geq\epsilon$ for all $f\in\F$ and for some fixed $\epsilon>0$
\end{enumerate}
Then $D$ is contained in the domain of normality of the corresponding SD family of $\F.$
\end{theorem}
\begin{proof}
Using Lemma~\ref{sec2,lemma2}\, $\F$ is locally bounded and from Theorem~\ref{sec2,thm2} the SD family is locally bounded on $D$ and so is normal on $D$ by Montel's theorem. The result then follows.
\end{proof}
We will require the following definition.
\begin{definition}
Let $f$ and $g$  be two meromorphic functions on the complex plane $\C.$ Then $f$ is conjugate to $g$ denoted by $f\sim g$, if there exists a Mobius transformation $\phi$ satisfying $\phi\circ f=g\circ\phi.$\\
Analogously one defines  two families $\F$ and $\G$ of meromorphic functions to be conjugate if there exist a family $\A$ of Mobius transformations satisfying $\phi\circ f=g\circ\phi$ for all $f\in\F,\,g\in\G$ and $\phi\in\A.$
\end{definition}
We further analyse the relationship between the domain of normality of a family of meromorphic functions and that  of its conjugate family. The resultant is no correlation justified with examples.
\begin{example}\label{eg4}
Consider $f_n(z)=e^{nz}$ and $g_n(z)=ne^z$ for all $n\in\N.$ Then $f_n\sim g_n$ under $\phi_n(z)=nz,\,n\in\N.$ The domain of normality of the former is $\C\setminus\{z:Re(z)=0\}$ while that of the latter is $\C.$
\end{example}
There is possibility of similar domains of normality as exhibited by the following example.
\begin{example}\label{eg7}
Consider $f_n(z)=e^{z+n}$ and $g_n(z)=e^z+n$ for all $n\in\N.$ Then $f_n\sim g_n$ under $\phi_n(z)=z+n,\,n\in\N.$ Both the families have domain of normality as $\C.$
\end{example}
Using the fact that the SD of a Mobius transformation is 0, one observes  if $\F$ and $\G$ are two families of meromorphic functions conjugate under the family $\A$ of Mobius transformations, then their corresponding SD family satisfies the important relation $(S_g\circ \phi)(\phi')^2=S_f$ for all $f\in\F,\,g\in\G$ and $\phi\in\A$ where $f\sim g$ under $\phi.$\\




\end{document}